\documentclass[12pt]{article}%
\usepackage{hyperref}
\usepackage{amsmath}
\usepackage{amsfonts}
\usepackage{amssymb}
\usepackage{graphicx}%
\newtheorem{theorem}{Theorem}[section]

\newtheorem{corollary}[theorem]{Corollary}

\newtheorem{example}[theorem]{Example}

\newtheorem{lemma}[theorem]{Lemma}

\newtheorem{proposition}[theorem]{Proposition}

\newenvironment{proof}[1][Proof]{\noindent\textbf{#1.} }{\ \rule{0.5em}{0.5em}}
\begin{document}

\title{Order-to-ring ideals via submultiplicative lattice seminorms in
lattice-ordered algebras}
\author{Karim Boulabiar\\{\small Laboratoire de Recherche LATAO}\\{\small D\'{e}partement de Mathematiques, Facult\'{e} des Sciences de Tunis}\\{\small Universit\'{e} de Tunis El Manar, 2092, El Manar, Tunisia}\\{\small Mathematical Institute, Leiden University}\\{\small P.O. Box 9512, 2300 RA Leiden, The Netherlands}
\and Rania Souissi\\{\small Laboratoire de Recherche LATAO}\\{\small D\'{e}partement de Mathematiques, Facult\'{e} des Sciences de Tunis}\\{\small Universit\'{e} de Tunis El Manar, 2092, El Manar, Tunisia}}
\maketitle

\begin{abstract}
This paper characterizes lattice-ordered algebras with the order-to-ring ideal
property, i.e., those in which every order ideal is a ring ideal. Our main
result shows that an Archimedean lattice-ordered algebra has this property if
and only if it is an $f$-algebra admitting a nil-faithful submultiplicative
lattice seminorm, where nil-faithful means that elements with zero seminorm
are nilpotent. As a consequence, we prove that an Archimedean semiprime
$f$-algebra has the order-to-ring ideal property if and only if it admits a
submultiplicative lattice norm. These results correct and extend earlier
incorrect work in the literature and yield a necessary and sufficient
condition for all orthomorphisms on an Archimedean vector lattice to be central.

\end{abstract}


\bigskip

\section{Introduction}

A lattice-ordered algebra (briefly, an $\ell$-algebra) is a mathematical
structure in which two fundamental features coexist, namely an algebraic
(ring) structure and a lattice-order. It is therefore natural to investigate
how these two aspects interact, in particular through the relationship between
order ideals and ring ideals. This interplay has been the subject of sustained
interest over the years, leading to a substantial body of literature (see, for
instance,
\cite{BT-1988,D-1968,HIJ-1961,HLS-1991,HP-1982,HP-1984,L-1987,MM-2013}).

In this paper, we focus on Archimedean $\ell$-algebras in which every order
ideal is automatically a ring ideal. For convenience, we refer to this
property as the \emph{order-to-ring ideal property}. A particularly
significant result in this direction was obtained by Basly and Triki in
\cite{BT-1988}. They show that an Archimedean $\ell$-algebra $A$ has the
order-to-ring ideal property if and only if it is an $f$-algebra whose
elements are all bounded in the sense that for each element $a\in A$, there
exists $\delta\in\left(  0,\infty\right)  $ such that $a^{2}\leq
\delta\left\vert a\right\vert $ (see also \cite{HLS-1991} for a related result
within the ZFC-Set Theory). As a notable consequence, they deduce that a
Banach $\ell$-algebra has the order-to-ring ideal property if and only if it
is an $f$-algebra (we refer to \cite{W-2017} for a comprehensive study of
Banach $\ell$-algebras). These results highlight the central role played by
$f$-algebras in this context and naturally motivate a more systematic
investigation of this class of structures.

Although these results provide a relevant overview of the topic under
consideration in this paper, the main motivation for this work also stems from
another statement in \cite{BT-1988}. It is claimed therein that an Archimedean
semiprime $f$-algebra enjoys the order-to-ring ideal property if and only if
it admits a lattice norm. Unfortunately, this statement turns out to be false
(its proof contains a subtle flaw), and somewhat surprisingly, it is refuted
by a relatively simple counterexample.

\begin{example}
\label{Example}Let $A$ be the set of all piecisewise polynomial real-valued
continuous functions on $\left[  0,\infty\right)  $. The pointwise addition,
multiplication, scalar multiplication, and order turn $A$ into an Archimedean
semiprime $f$-algebra. Clearly, $A$ contains unbounded functions, and
consequently an order ideal of $A$ need not be a ring ideal. Nevertheless, the
formula%
\[
\left\Vert a\right\Vert =\sup\left\{  \left\vert e^{-r}a\left(  r\right)
\right\vert :r\in\left[  0,\infty\right)  \right\}  \quad\text{for all }a\in
A
\]
defines a lattice norm on $A$.
\end{example}

In fact, our paper addresses several related objectives. First, we correct the
theorem of Basly and Triki mentioned above by identifying the additional
condition required for the equivalence to hold. Indeed, we show that an
Archimedean semiprime $f$-algebra has the order-to-ring ideal property if and
only if it admits a submultiplicative lattice norm.

This result is in fact a special case of a more general characterization of
Archimedean $f$-algebras with the order-to-ring ideal property in terms of
nil-faithful submultiplicative lattice seminorms, where nil-faithfulness means
that the seminorm vanishes only on nilpotent elements. More precisely, we
prove that an arbitrary (not necessarily semiprime) Archimedean $\ell$-algebra
has the order-to-ring ideal property if and only if it admits a nil-faithful
submultiplicative lattice seminorm.

In addition, we extend the validity of the aforementioned result by Basly and
Triki on Banach $\ell$-algebras to the broader setting of (not necessarily
complete) normed $\ell$-algebras.

Finally, we will apply the main result to orthomorphisms on vector lattices.
We know that if $L$ is an Archimedean vector lattice, then $\mathrm{Orth}%
\left(  L\right)  $, the set of all orthomorphisms on $L$, is an Archimedean
semiprime $f$-algebra \cite{Z-1983}. As a consequence of our main result, we
obtain that, if $Z\left(  L\right)  $ denotes the center of $L$, then
$\mathrm{Orth}\left(  L\right)  =Z\left(  L\right)  $ if and only if
$\mathrm{Orth}\left(  L\right)  $ can be equipped with a submultiplicative
seminorm. This may be regarded as a modest contribution to the long-standing
open problem, viz., characterizing Archimedean vector lattices $L$ satisfying
$\mathrm{Orth}\left(  L\right)  =Z\left(  L\right)  $ (we shall discuss a
partial answer, yet an important one, due to Wickstead \cite{W-1977}, in the
last section).

The paper is organized as follows. We begin by recalling the necessary
preliminaries and introducing the notion of bounded elements, which plays a
key role in our approach. We then establish our main results and derive
several consequences, culminating in an application to orthomorphisms on
Archimedean vector lattices.

\section{Some preliminaries}

We assume familiarity with the elementary theory of vector lattices (also
known as Riesz spaces). For a comprehensive treatment of this subject, we
refer the reader to the classical monograph \cite{AB-2006-PO} by Aliprantis
and Burkinshaw.

Throughout, all vector lattices considered are real and Archimedean. Recall
that a vector lattice $L$ is said to be \emph{Archimedean} if, for every
$x,y\in L$,%
\[
nx\leq y\text{ for all }n\in\left\{  1,2,...\right\}  \text{\quad
implies\quad}x\leq0.
\]
A vector subspace $J$ of a vector lattice $L$ is called an \emph{order ideal}
of $L$ if it is \emph{solid} in $L$, i.e., if%
\[
\left\vert x\right\vert \leq\left\vert y\right\vert \text{ in }L\text{ and
}y\in J\text{\quad imply\quad}x\in J.
\]
It is not hard to verify that the vector subspace $J$ of $L$ is an order ideal
of $L$ if and only if $J$ is a \emph{vector sublattice} of $L$, i.e.,%
\[
\left\vert x\right\vert \in J\text{\quad for all }x\in J,
\]
and%
\[
0\leq x\leq y\text{ in }L\text{ and }x\in y\text{\quad imply\quad}x\in J.
\]
For instance, if $Y$ is a non-empty subset of the vector lattice $L$ then its
\emph{disjoint complement} defined by%
\[
Y^{d}=\left\{  x\in L:\left\vert x\right\vert \wedge\left\vert y\right\vert
=0\text{ for all }y\in Y\right\}
\]
is an order ideal of $L$. Furthermore, for each $x\in L$, the set%
\[
L_{x}=\left\{  y\in L:\left\vert y\right\vert \leq\delta\left\vert
x\right\vert \text{ for some }\delta\in\left(  0,\infty\right)  \right\}
\]
is also an order ideal of $L$. It is, in fact, the smallest order ideal
containing $x$, and is referred to as the \emph{principal order ideal} of $L$
generated by $x$.

The next paragraph provides a brief overview of the theory of $f$-algebras. A
detailed exposition can be found in the final chapter of Zaanen's monograph
\cite{Z-1983}, while a particularly thorough and insightful treatment is given
in de Pagter's doctoral thesis \cite{P-1981}.

An associative real algebra $A$ is called a \emph{lattice-ordered algebra}
(briefly, an $\ell$\emph{-algebra}) if it is simultaneously a vector lattice
whose positive cone%
\[
A^{+}=\left\{  a\in A:0\leq a\right\}
\]
is \emph{closed under multiplication}, i.e.,%
\[
ab\in A^{+}\text{\quad for all }a,b\in A^{+},
\]
or, equivalently,%
\[
\left\vert ab\right\vert \leq\left\vert a\right\vert \left\vert b\right\vert
\text{\quad for all }a,b\in A.
\]
As usual, a vector subspace $J$ of the $\ell$-algebra $A$ is called a
\emph{ring ideal} of $A$ if%
\[
ax\in J\text{ and }xa\in J\text{\quad for all }a\in A\text{ and }x\in J.
\]
The $\ell$-algebra $A$ is called an $f$\emph{-algebra} if the disjoint
complement%
\[
\left\{  a\right\}  ^{d}=\left\{  x\in A:\left\vert a\right\vert
\wedge\left\vert x\right\vert =0\right\}
\]
of any element $a\in A$ is a ring ideal of $A$. A short moment's thought
reveals that, if $A$ is an $f$-algebra, then%
\[
ab=0\text{\quad for all }a,b\in A\text{ with }a\wedge b=0.
\]
It follows that%
\[
\left\vert ab\right\vert =\left\vert a\right\vert \left\vert b\right\vert
\text{\quad for all }a,b\in A.
\]
Moreover,%
\[
a^{+}a^{-}=a^{-}a^{+}=0\text{\quad for all }a\in A,
\]
so,%
\[
a^{+}a=\left(  a^{+}\right)  ^{2}\text{ and }a^{2}=\left\vert a\right\vert
^{2}\text{\quad for all }a\in A.
\]
In particular, every square in an $f$-algebra is positive. Moreover, a
fundamental result in the theory of $f$-algebras asserts that every
Archimedean $f$-algebra is commutative (see, e.g., \cite[Theorem
2.56]{AB-2006-PO}). Consequently, if $N\left(  A\right)  $ denotes the set of
all nilpotent elements of an $\ell$-algebra $A$, i.e.,%
\[
N\left(  A\right)  =\left\{  a\in A:a^{n}=0\right\}  ,
\]
then the nilradical of an Archimedean $f$-algebra $A$ coincides with $N\left(
A\right)  $.

We now record some basic properties of the nilpotent elements in an
Archimedean $f$-algebra (proofs can be found in \cite[Proposition
10.2]{P-1981}).

\begin{lemma}
\label{Pagter}Let $A$ be an Archimedean $f$-algebra. Then the following
assertions hold.

\begin{enumerate}
\item[\emph{(i)}] $N\left(  A\right)  =\left\{  a\in A:a^{2}=0\right\}
=\left\{  a\in A:ax=0\text{ for all }x\in A\right\}  $.

\item[\emph{(ii)}] $ab\in N\left(  A\right)  ^{d}$ for all $a,b\in A$.
\end{enumerate}
\end{lemma}

These results will be used to establish the following technical lemma, which
concludes this preliminary section.

\begin{lemma}
\label{Fundamental}Let $\delta\in\left(  0,\infty\right)  $ and $a$ be an
element of an Archimedean $f$-algebra $A$. Then $a^{2}\leq\delta\left\vert
a\right\vert $ if and only if $\left\vert ax\right\vert \leq\delta\left\vert
x\right\vert $ for all $x\in A$.
\end{lemma}

\begin{proof}
The sufficiency being trivial, we prove the necessity. Assume that $a^{2}%
\leq\delta\left\vert a\right\vert $ and let $x\in A$. Since $\left(
a^{2}-\delta\left\vert a\right\vert \right)  ^{+}=0$, we get%
\[
\left\vert ax\right\vert \left(  \left\vert ax\right\vert -\delta\left\vert
x\right\vert \right)  ^{+}=\left(  a^{2}-\delta\left\vert a\right\vert
\right)  ^{+}x^{2}=0.
\]
Accordingly,%
\[
0\leq\left[  \left(  \left\vert ax\right\vert -\delta\left\vert x\right\vert
\right)  ^{+}\right]  ^{2}\leq\left\vert ax\right\vert \left(  \left\vert
ax\right\vert -\delta\left\vert x\right\vert \right)  ^{+}=0.
\]
It follows that%
\[
\left(  \left\vert ax\right\vert -\delta\left\vert x\right\vert \right)
^{+}\in N\left(  A\right)  .
\]
Moreover, using Lemma \ref{Pagter} $\mathrm{(ii)}$ and the inequalities%
\[
0\leq\left(  \left\vert ax\right\vert -\delta\left\vert x\right\vert \right)
^{+}\leq\left\vert ax\right\vert ,
\]
we derive that%
\[
\left(  \left\vert ax\right\vert -\delta\left\vert x\right\vert \right)
^{+}\in N\left(  A\right)  ^{d}.
\]
But then $\left(  \left\vert ax\right\vert -\delta\left\vert x\right\vert
\right)  ^{+}=0$, and the lemma follows.
\end{proof}

\section{Bounded elements}

An element $a$ of an Archimedean $f$-algebra is said to be \emph{bounded} if%
\[
a^{2}\leq\delta\left\vert a\right\vert \text{\quad for some }\delta\in\left(
0,\infty\right)  .
\]
Following the notation of \cite{GJ-1976}, we denote by $A^{\ast}$ the set all
bounded elements of $A$, i.e.,%
\[
A^{\ast}=\left\{  a\in A:a^{2}\leq\delta\left\vert a\right\vert \text{ for
some }\delta\in\left(  0,\infty\right)  \right\}  .
\]
By Lemma \ref{Fundamental}, this set can equivalently be characterized as%
\[
A^{\ast}=\left\{  a\in A:\left\vert ax\right\vert \leq\delta\left\vert
x\right\vert \text{ for some }\delta\in\left(  0,\infty\right)  \text{ and all
}x\in A\right\}  .
\]
Consequently, the following result can be obtained by a routine computation,
and we therefore omit the proof.

\begin{proposition}
\label{Bounded}If $A$ is an Archimedean $f$-algebra, then $A^{\ast}$ is both
an order ideal and a subalgebra of $A$.
\end{proposition}

We derive in particular from Proposition \ref{Bounded} that, if $A\ $is an
Archimedean $f$-algebra, then $A^{\ast}$ is an Archimedean $f$-algebra in its
own right.

Recall now that a seminorm $p$ on a vector lattice $L$ is called a
\emph{lattice} (or, \emph{Riesz}) seminorm if%
\[
p\left(  x\right)  \leq p\left(  y\right)  \text{ whenever }x,y\in L\text{
satisfy }\left\vert x\right\vert \leq\left\vert y\right\vert .
\]
A seminorm $p$ on $L$ is a lattice seminorm if and only if%
\[
p\left(  \left\vert x\right\vert \right)  =p\left(  x\right)  \text{\quad for
all }x\in L
\]
and%
\[
p\left(  x\right)  \leq p\left(  y\right)  \text{\quad for all }x,y\in L\text{
with }0\leq x\leq y.
\]
On the other hand, a seminorm $p$ on an associative real algebra $A$ is said
to be \emph{submultiplicative} if%
\[
p\left(  ab\right)  \leq p\left(  a\right)  p\left(  b\right)  \text{\quad for
all }a,b\in A.
\]
We also say that a seminorm $p$ on $A$ is \emph{nil-faithful} if $p$ vanishes
only on nilpotent elements, that is, if $a\in A$ and $p\left(  a\right)  =0$
then $a$ is necessarily nilpotent. Trivially, any norm on $A$ is nil-faithful.
A fundamental example of a nil-faithful submultiplicative lattice semi norms
on the bounded elements of an Archimedean $f$-algebra is presented next.

For the rest of this section, we fix an Archimedean $f$-algebra $A$ and
consider the set%
\[
C\left(  A\right)  =\left\{  a\in A:a^{2}\leq\left\vert a\right\vert \right\}
.
\]
From Lemma \ref{Fundamental}, it follows straightforwardly that%
\[
C\left(  A\right)  =\left\{  a\in A:\left\vert ax\right\vert \leq\left\vert
x\right\vert \text{ for all }x\in A\right\}  .
\]
The following lemma collects elementary properties of $C\left(  A\right)  $.

\begin{lemma}
\label{Convex}Let $A$ be an Archimedean $f$-algebra. Then the following hold.

\begin{enumerate}
\item[\emph{(i)}] $C\left(  A\right)  $ is convex, i.e.,%
\[
\delta a+\left(  1-\delta\right)  b\in C\left(  A\right)  \text{ for all
}a,b\in C\left(  A\right)  \text{ and }\delta\in\left[  0,1\right]  .
\]

\item[\emph{(ii)}] $C\left(  A\right)  $ is solid.

\item[\emph{(iii)}] $C\left(  A\right)  $ is idempotent, i.e.,%
\[
ab\in C\left(  A\right)  \quad\text{for all }a,b\in C\left(  A\right)  .
\]

\end{enumerate}
\end{lemma}

\begin{proof}
$\mathrm{(i)}$ Let $a,b\in C\left(  A\right)  $ and $\delta\in\left[
0,1\right]  $. For every $x\in A$, we have%
\begin{align*}
\left\vert \left(  \delta a+\left(  1-\delta\right)  b\right)  x\right\vert
&  \leq\delta\left\vert ax\right\vert +\left(  1-\delta\right)  \left\vert
bx\right\vert \\
&  \leq\delta\left\vert x\right\vert +\left(  1-\delta\right)  \left\vert
x\right\vert =\left\vert x\right\vert .
\end{align*}
We derive that $C\left(  A\right)  $ is convex.

$\mathrm{(ii)}$ Let $a\in C\left(  A\right)  $ and $b\in A$ such that
$\left\vert b\right\vert \leq\left\vert a\right\vert $. If $x\in A$ then%
\[
\left\vert bx\right\vert =\left\vert b\right\vert \left\vert x\right\vert
\leq\left\vert a\right\vert \left\vert x\right\vert =\left\vert ax\right\vert
\leq\left\vert x\right\vert .
\]
It follows that $b\in C\left(  A\right)  $, so $C\left(  A\right)  $ is solid.

$\mathrm{(iii)}$ If $a,b\in C\left(  A\right)  $ then%
\[
\left(  ab\right)  ^{2}=a^{2}b^{2}\leq\left\vert a\right\vert \left\vert
b\right\vert =\left\vert ab\right\vert .
\]
Hence, $C\left(  A\right)  $ is idempotent and the proof is complete.
\end{proof}

It is well-known that the Minkowski functional (gauge) of $C\left(  A\right)
$ is defined by%
\[
p_{\infty}\left(  a\right)  =\inf\left\{  \delta\in\left(  0,\infty\right)
:\frac{1}{\delta}a\in C\left(  A\right)  \right\}  \quad\text{for all }a\in
A,
\]
where, by convention, $\inf\emptyset=\infty$. Clearly, $p_{\infty}\left(
a\right)  $ can equivalently be written as%
\[
p_{\infty}\left(  a\right)  =\inf\left\{  \delta\in\left(  0,\infty\right)
:a^{2}\leq\delta\left\vert a\right\vert \right\}  \quad\text{for all }a\in A.
\]
It follows that%
\[
A^{\ast}=\left\{  a\in A:p_{\infty}\left(  a\right)  <\infty\right\}  .
\]
Thus, $p_{\infty}$ is a well-defined real-valued function on $A^{\ast}$. In
view of the properties of $C\left(  A\right)  $ listed in Lemma \ref{Convex},
and by the remark following \cite[Definition 8.45]{AB-2006-IDA} together with
\cite[Lemma 1.2]{M-1952}, we derive that $p_{\infty}$ is a submultiplicative
lattice seminorm on $A^{\ast}$. As the next proposition shows, $p_{\infty}$
enjoys an additional property.

\begin{proposition}
\label{Seminorm}Let $A$ be an Archimedean $f$-algebra. Then $p_{\infty}$ is
nil-faithful submultiplicative lattice seminorm on $A^{\ast}$.
\end{proposition}

\begin{proof}
It remains to show that $p_{\infty}$ is nil-faithful on $A^{\ast}$. Let $a\in
A^{\ast}$ satisfy $p_{\infty}\left(  a\right)  =0$. Then,%
\[
\inf\left\{  \delta\in\left(  0,\infty\right)  :a^{2}\leq\delta\left\vert
a\right\vert \right\}  =0.
\]
Therefore, for every $n\in\left\{  1,2,...\right\}  $, there exists $\delta
\in\left(  0,\infty\right)  $ such that%
\[
a^{2}\leq\delta\left\vert a\right\vert \quad\text{and\quad}\delta<\frac{1}%
{n}.
\]
It follows that%
\[
na^{2}\leq\left\vert a\right\vert \quad\text{for all }n\in\left\{
1,2,...\right\}  .
\]
Since $A$ is Archimedean, this yields $a^{2}=0$ and completes the proof.
\end{proof}

We conclude this section with the following simple observation, which will be
used later with no specific mention. Using \cite[Lemma 5.50]{AB-2006-IDA}, we have%

\[
C\left(  A\right)  =\left\{  a\in A:a^{2}\leq\left\vert a\right\vert \right\}
=\left\{  a\in A:p_{\infty}\left(  a\right)  \leq1\right\}  .
\]
In particular, for every $a\in A$ and $\delta\in\left[  0,\infty\right)  $, we
have%
\[
p_{\infty}\left(  a\right)  \leq\delta\text{\quad if and only\quad if }%
a^{2}\leq\delta\left\vert a\right\vert .
\]

\section{Main results}

We begin this section with the following lemma, which shows that, if $A$ is an
Archimedean $f$-algebra, then any nil-faithful submultiplicative lattice
seminorm on $A$ ---if one exists--- is at least as large as $p_{\infty}$.

\begin{lemma}
\label{Gauge}A lattice seminorm $p$ on an Archimedean $f$-algebra $A$ is
nil-faithful and submultiplicative if and only if%
\[
p_{\infty}\left(  a\right)  \leq p\left(  a\right)  \quad\text{for all }a\in
A.
\]

\end{lemma}

\begin{proof}
Assume that $p$ is nil-faithful and submultiplicative and let $a\in A$ with
$a>0$. Fix $n\in\left\{  1,2,...\right\}  $ and define%
\[
\delta=\frac{1}{n}+p\left(  a\right)  \text{\quad and\quad}b=\left(
a^{2}-\delta a\right)  ^{+}.
\]
Suppose, for contradiction, that $b>0$. By Lemma \ref{Pagter}, we have%
\[
0<b\leq a^{2}\in N\left(  A\right)  ^{d},
\]
and hence $b\notin N\left(  A\right)  $. Furthermore,%
\[
0<b^{3}=b^{2}\left(  a^{2}-\delta a\right)  =b^{2}a^{2}-\delta b^{2}a,
\]
which implies that $b^{2}a\notin N\left(  A\right)  $. Since $p$ is
nil-faithful, it follows that $p\left(  b^{2}a\right)  >0$. On the other hand,
using the monotony and the submultiplicativity of $p$, together with the
inequality%
\[
0\leq\delta b^{2}a\leq b^{2}a^{2},
\]
we obtain%
\[
0\leq\delta p\left(  b^{2}a\right)  \leq p\left(  b^{2}a^{2}\right)  \leq
p\left(  b^{2}a\right)  p\left(  a\right)  .
\]
We deduce that $\delta\leq p\left(  a\right)  $, contradicting the definition
of $\delta$. Therefore, $b=0$, and so%
\[
n\left(  a^{2}-p\left(  a\right)  a\right)  \leq a\quad\text{for all }%
n\in\left\{  1,2,...\right\}  .
\]
Since $A$ is Archimedean, we conclude that $a^{2}\leq p\left(  a\right)  a$.
By definition of $p_{\infty}$, it follows that $p_{\infty}\left(  a\right)
\leq p\left(  a\right)  $ and necessity follows.

Conversely, assume that $p_{\infty}\left(  a\right)  \leq p\left(  a\right)  $
for all $a\in A$. Hence, $a^{2}\leq p\left(  a\right)  \left\vert a\right\vert
$ for all $a\in A$. Obviously, if $p\left(  a\right)  =0$ for some $a\in A$
then $a^{2}=0$, and so $a\in N\left(  A\right)  $. This means that $p$ is
nil-faithful. We now claim that $p$ is submultiplicative. To this end, let
$a,b\in A$ and observe that from Lemma \ref{Fundamental} it follows that%
\[
\left\vert ab\right\vert \leq p\left(  a\right)  \left\vert b\right\vert .
\]
Therefore,%
\[
p\left(  ab\right)  \leq p\left(  p\left(  a\right)  \left\vert b\right\vert
\right)  =p\left(  a\right)  p\left(  b\right)  .
\]
This shows that $p$ is submultiplicative, as desired.
\end{proof}

We are in position at this point to state and prove the central result of this paper.

\begin{theorem}
\label{Main}Let $A$ be an Archimedean $\ell$-algebra. Then the following
conditions are equivalent.

\begin{enumerate}
\item[\emph{(i)}] $A$ has the order-to-ring ideal property

\item[\emph{(ii)}] $A$ is an $f$-algebra and $A=A^{\ast}$.

\item[\emph{(iii)}] $A$ is an $f$-algebra with a nil-faithful
submultiplicative lattice seminorm.
\end{enumerate}
\end{theorem}

\begin{proof}
$\mathrm{(i)\Rightarrow(ii)}$ This implication was already proved by Basly and
Triki in \cite{BT-1988}. For completeness, we include the details here. For
each $a\in A$, the disjoint complement $\left\{  a\right\}  ^{d}$ is an order
ideal of $A$, and hence, by assumption, a ring ideal of $A$. This implies that
$A$ is an $f$-algebra. Now consider%
\[
A_{a}=\left\{  x\in A:\left\vert x\right\vert \leq\delta\left\vert
a\right\vert \text{ for some }\delta\in\left(  0,\infty\right)  \right\}  .
\]
This is the principal order ideal of $A\ $generated by $a$. By hypothesis,
$A_{a}$ is a ring ideal of $A$. Since $a\in A_{a}$, it follows that $a^{2}\in
A_{a}$. Therefore, there exists $\delta\in\left(  0,\infty\right)  $ such that
$a^{2}\leq\delta\left\vert a\right\vert $, which yields that $A=A^{\ast}$.

$\mathrm{(ii)\Rightarrow(iii)}$ This follows by considering $p_{\infty}$ and
applying Proposition \ref{Gauge}.

$\mathrm{(iii)\Rightarrow(i)}$ Let $p$ be a nil-faithful submultiplicative
lattice seminorm on $A$, and let $J$ be an order ideal of $A$. We claim that
$J$ is a ring ideal of $A$. To this end, choose $a\in A$ and $x\in J$. We must
prove that $ax\in J$. By Lemma \ref{Gauge}, we have%
\[
a^{2}\leq p\left(  a\right)  \left\vert a\right\vert .
\]
Using Lemma \ref{Fundamental}, we get%
\[
\left\vert ax\right\vert \leq p\left(  a\right)  \left\vert x\right\vert ,
\]
which shows that $ax\in J$. Thus, $J$ is an ring ideal of $A$, which is the
desired result.
\end{proof}

As usual, an $\ell$-algebra equipped with a submultiplicative lattice seminorm
is referred to as a \emph{normed }$\ell$\emph{-algebra. }Corollary 4 in
\cite{BT-1988} shows that a Banach $\ell$-algebra (i.e., a normed $\ell
$-algebra whose norm is complete) has the order-to-ring ideal property if and
only if it is an $f$-algebra. The following result extends this fact to
general normed $\ell$-algebras.

\begin{corollary}
\label{Banach extension}Let $A$ be a normed $\ell$-algebra. Then the following
conditions are equivalent.

\begin{enumerate}
\item[\emph{(i)}] $A$ has the order-to-ring ideal property.

\item[\emph{(ii)}] $A$ is an $f$-algebra with $A=A^{\ast}$.

\item[\emph{(iii)}] $A$ is an $f$-algebra.
\end{enumerate}
\end{corollary}

\begin{proof}
$\mathrm{(i)\Rightarrow(ii)}$ Since $A$ carries a lattice norm, it is
Archimedean. By Theorem \ref{Main}, $A$ is an $f$-algebra.

$\mathrm{(ii)\Rightarrow(iii)}$ Trivial.

$\mathrm{(iii)\Rightarrow(i)}$ As observed earlier, every norm is a
nil-faithful seminorm. Therefore, $A$ satisfies the condition $\mathrm{(iii)}
$ in Theorem \ref{Main}, and so $A$ has the order-to-ring ideal property.
\end{proof}

An $\ell$-algebra $A$ is called \emph{semiprime} if $N\left(  A\right)
=\left\{  0\right\}  $, that is, if $0$ is the only nilpotent element of $A$.
The following corollary gives the correct formulation of Theorem 5 in
\cite{BT-1988}, which we show in the introduction that is false.

\begin{corollary}
\label{Semiprime}Let $A$ be an Archimedean semiprime $\ell$-algebra. Then the
following conditions are equivalent.

\begin{enumerate}
\item[\emph{(i)}] $A$ has the order-to-ring ideal property.

\item[\emph{(ii)}] $A=A^{\ast}$.

\item[\emph{(iii)}] $A$ has a submultiplicative lattice norm.
\end{enumerate}
\end{corollary}

\begin{proof}
$\mathrm{(i)\Rightarrow(ii)}$ This follows directly from Theorem \ref{Main}.

$\mathrm{(ii)\Rightarrow(iii)}$ By Theorem \ref{Main}, $A$ has a nil-faithful
submultiplicative lattice seminorm $p$. Since $A$ is semiprime, it follows
that $p$ is a norm.

$\mathrm{(iii)\Rightarrow(i)}$ This is also a straightforward consequence of
Theorem \ref{Main}.
\end{proof}

\section{An application to orthomorphisms}

All operators considered in this last section are assumed to be linear.

Let $L$ be an Archimedean vector lattice. Recall that the set $\mathcal{L}%
\left(  L\right)  $ of all operators on $L$ is a (real) vector space with
respect to the pointwise addition and scalar multiplication. Moreover, the
operator composition makes $\mathcal{L}\left(  L\right)  $ a unital
associative algebra, whose unit element is the identity operator on $L$,
denoted by $I_{L}$.

An operator $T$ on $L$ is said to be \emph{order bounded} if, for every $u\in
L$ there exists $v\in L$ such that%
\[
\left\vert T\left(  x\right)  \right\vert \leq v\quad\text{for all }x\in
L\text{ with }\left\vert x\right\vert \leq u.
\]
An order bounded operator $T$ on $L$ is referred to as an \emph{orthomorphism}
if, for every $x,y\in L$,%
\[
\left\vert x\right\vert \wedge\left\vert y\right\vert =0\quad
\text{implies\quad}\left\vert T\left(  x\right)  \right\vert \wedge\left\vert
y\right\vert =0.
\]
Notice that $I_{L}$ is an orthomorphism on $L$. We denote by $\mathrm{Orth}%
\left(  L\right)  $ the set of all orthomorphisms on $L$, and proceed to study
its lattice-order algebraic structure.

One can readily checked that $\mathrm{Orth}\left(  L\right)  $ is a subalgebra
of $\mathcal{L}\left(  L\right)  $. Furthermore, an order can be introduced on
$\mathrm{Orth}\left(  L\right)  $ in a pointwise manner by declaring that%
\[
T\leq S\text{ in }\mathrm{Orth}\left(  L\right)  \text{\quad if and only
if\quad}S\left(  x\right)  \leq T\left(  x\right)  \text{ for all }x\in
L^{+}.
\]
It turns out that, equipped with this order, $\mathrm{Orth}\left(  L\right)  $
is an Archimedean semiprime $f$-algebra. The absolute value in $\mathrm{Orth}%
\left(  L\right)  $ enjoys a remarkable and highly useful property.
Specifically, if $T\in\mathrm{Orth}\left(  L\right)  $ then%
\[
\left\vert T\right\vert \left(  \left\vert x\right\vert \right)  =\left\vert
T\left(  \left\vert x\right\vert \right)  \right\vert =\left\vert T\left(
x\right)  \right\vert \quad\text{for all }x\in L.
\]
The principal order ideal of $\mathrm{Orth}\left(  L\right)  $ generated by
$I_{L}$ is denoted by $Z\left(  L\right)  $ and called the \emph{center} of
$L$. Accordingly, any operator in $Z\left(  L\right)  $ is called
\emph{central}. Thus, if $T\in\mathrm{Orth}\left(  L\right)  $ then $T$ is
central if and only if%
\[
\left\vert T\right\vert \leq\delta I_{L}\quad\text{for some }\delta\in\left(
0,\infty\right)  .
\]
For further information on the theory of orthomorphisms, the reader is
encouraged to consult \cite{P-1981} or \cite{Z-1983}.

As already observed in the introduction, characterizing Archimedean vector
lattices $L$ satisfying $\mathrm{Orth}\left(  L\right)  =Z\left(  L\right)  $
remains a long-standing open problem in the theory of operators on vector
lattices. A partial, yet fundamental, answer is provided by a classical result
of Wickstead \cite{W-1977}, namely that this equality holds for every Banach
lattice $L$. This important result has been extensively used by Basly and
Triki \cite{BT-1988} to obtain the Banach version of Corollary
\ref{Banach extension}. In what follows, we reverse this viewpoint by
recovering Wickstead's theorem as a direct consequence of this corollary.

First, let's state the following theorem, which follows directly from
Corollary \ref{Semiprime} since $\mathrm{Orth}\left(  L\right)  $ is an
Archimedean semiprime $f$-algebra.

\begin{theorem}
\label{Open}Let $L$ be an Archimedean vector lattice. Then $\mathrm{Orth}%
\left(  L\right)  =Z\left(  L\right)  $ if and only if $\mathrm{Orth}\left(
L\right)  $ has a submultiplicative lattice norm.
\end{theorem}

We proceed to show how we can obtain the aforementioned Wickstead's theorem as
a consequence of Theorem \ref{Open}.

\begin{corollary}
\emph{(Wickstead)} If $L$ is a Banach lattice, then $\mathrm{Orth}\left(
L\right)  =Z\left(  L\right)  $.
\end{corollary}

\begin{proof}
In view of Theorem \ref{Open}, it suffices to prove that $\mathrm{Orth}\left(
L\right)  $ is equipped with a submultiplicative lattice norm. Let $B\left(
L\right)  $ denote the Banach algebra $B\left(  L\right)  $ of all continuous
operators on $L$ (see, e.g., \cite[Section \S 1]{BD-1973}). The norm on
$B\left(  L\right)  $ is defined by%
\[
\left\Vert T\right\Vert =\sup\left\{  \left\Vert T\left(  x\right)
\right\Vert :x\in L\text{ and }\left\Vert x\right\Vert \leq1\right\}
\quad\text{for all }T\in L\left(  B\right)  .
\]
Since any positive operator on $L$ is continuous (see \cite[Theorem
4.3]{AB-2006-PO}), every operator in $\mathrm{Orth}\left(  L\right)  $ is
continuous, as difference between two positive operators on $L$. It readily
follows that $\mathrm{Orth}\left(  L\right)  $ is a subalgebra of $L\left(
B\right)  $. Therefore, the restriction of the above supremum norm on
$\mathrm{Orth}\left(  L\right)  $ is a submultiplicative norm. It remains to
show that it is a lattice norm. Let $S,T\in\mathrm{Orth}\left(  L\right)  $
such that $\left\vert S\right\vert \leq\left\vert T\right\vert $ in
$\mathrm{Orth}\left(  L\right)  $. Hence, for every $x\in L$, we have%
\[
\left\vert T\left(  x\right)  \right\vert =\left\vert T\right\vert \left(
\left\vert x\right\vert \right)  \leq\left\vert S\right\vert \left(
\left\vert x\right\vert \right)  =\left\vert S\left(  x\right)  \right\vert
\]
(where we use, e.g., \cite[Theorem 140.4]{Z-1983}). Accordingly, $\left\Vert
T\right\Vert \leq\left\Vert S\right\Vert $ and the corollary follows.
\end{proof}

We conclude this paper by presenting a few observations that we believe are
worth mentioning and that may be of interest for further discussion.

\begin{enumerate}
\item In \cite[Theorem 5]{BT-1988}, the reader will observe yet another
necessary and sufficient condition on an Archimedean semiprime $f$-algebra for
the order-to-ring ideal property, namely that $A$ admits an $M$-norm (see
Definition 4.20 in \cite{AB-2006-PO}). However, this equivalence is not
correct either. Indeed, the Archimedean semiprime $f$-algebra in Example
\ref{Example} contains unbounded functions, although it admits an $M$-norm
(the norm given in the example itself). As a matter of fact, the equivalence
becomes valid when replacing the condition of having an $M$-norm by that of
having a submultiplicative $M$-norm. This follows from the fact that the
submultiplicative lattice norm $p_{\infty}$ on any Archimedean semiprime
$f$-algebra is an $M$-norm.

\item There exists an alternative characterization of nil-faithful seminorms
on $f$-algebras that may be of independent interest. A seminorm $p$ on an
$f$-algebra $A$ is nil-faithful if and only if $N\left(  A\right)  $ is
$p$-closed in $A$. First, suppose that $N\left(  A\right)  $ is $p$-closed in
$A$ and let $a\in A$ satisfy $p\left(  a\right)  =0$. Define $a_{n}=0$ for all
$n\in\left\{  1,2,...\right\}  $. Then $a_{n}\in N\left(  A\right)  $ for all
$n\in\left\{  1,2,...\right\}  $, and the sequence $\left(  a_{n}\right)
_{n=1}^{\infty}$ is converges to $a$ with respect to $p$. Since $N\left(
A\right)  $ is $p$-closed, it follows that $a\in N\left(  A\right)  $, showing
that $p$ is nil-faithful. Conversely, assume that $p$ is nil-faithful and let
us show that $N\left(  A\right)  $ is $p$-closed. Pick a sequence $\left(
a_{n}\right)  _{n=1}^{\infty}$ in $N\left(  A\right)  $ that converges to some
$a\in A$ with respect to $p$. For each $n\in\left\{  1,2,...\right\}  $, we
have%
\[
0\leq p\left(  a^{2}\right)  =p\left(  \left(  a_{n}-a\right)  ^{2}\right)
\leq p\left(  a_{n}-a\right)  ^{2}.
\]
Letting $n\rightarrow\infty$, we get $p\left(  a^{2}\right)  =0$ and so
$a^{2}\in N\left(  A\right)  $. This implies that $a\in N\left(  A\right)  $,
as required.

\item As usual, let $C\left(  X\right)  $ denote the set of all real-valued
continuous functions on a topological space $X$. It is well known (and easy to
verify) that $C\left(  X\right)  $ is an Archimedean semiprime $f$-algebra
with respect to the pointwise algebraic operations and order \cite{AB-2006-PO}%
. Hence, it follows from Corollary \ref{Semiprime} that $C\left(  X\right)  $
admits a submultiplicative lattice norm if and only if $C\left(  X\right)  $
contains no unbounded functions. This result can also be derived as a
consequence of a theorem of Jorosz \cite{J-2015}, who proved a similar
statement without assuming that the norm is a lattice norm.

\item The reader will recognize in Lemma \ref{Gauge} an echo of a classical
theorem of Kaplansky \cite{K-1949}, namely, for a compact space $X$, the
supremum norm on $C\left(  X\right)  $ is minimal among all submultiplicative norms.
\end{enumerate}

\end{document}